\documentclass{scrartcl}

\usepackage[
theoremdefs,
final,
]{allergy}

\usepackage{unixode}

\usepackage[numbers]{natbib}
\usepackage{url}

\usepackage{etrees}

\usepackage{mdframed}

\newcommand*\demph[1]{\textbf{\emph{#1}}}

\newcommand*\mul{\,}

\newcommand*\NN{\mathbf{N}}
\newcommand*\RR{\mathbf{R}}

\newcommand*\sym{{S}}

\NewDocumentCommand\GL{O{n}}{\mathsf{GL}(#1)}
\NewDocumentCommand\affrel{O{a}}{\stackrel{#1}\rightsquigarrow}

\NewDocumentCommand\Poly{O{k}}{\operatorname{P}_{#1}}

\newcommand*\Lin{\mathcal{L}}

\NewDocumentCommand\Graph{O{}O{}}{\Gamma^{#1}_{#2}}
\NewDocumentCommand\Gv{O{}O{}}{\spangle{\Graph[#1][#2]}}
\newcommand*\Tree{T}
\newcommand*\Tc{M}
\newcommand*\Tv{\spangle{T}}

\newcommand*\eldiff{\mathcal{F}}

\newcommand*\dappl[2]{#1\bracket{#2}}

\NewDocumentCommand\eldappl{mmO{}}{\dappl{#1}{#2}}
\NewDocumentCommand\eldfappl{mmO{}}{\dappl{\eldiff_{#3}\paren*{#1}}{#2}}

\newcommand*\graphprop[1]{\mathcal{#1}}
\newcommand*\Roots{\graphprop{R}}
\newcommand*\Parents{\graphprop{P}}
\newcommand*\Edges{\graphprop{E}}
\newcommand*\Vertices{\graphprop{V}}

\newcommand*\dgr{\operatorname{dep}}

\newcommand*\rvf[1]{{\mathfrak{X}_0}(\RR^{#1})}

\NewDocumentCommand\pvf{O{n}}{X^{#1}}
\NewDocumentCommand\psf{O{n}}{F^{#1}}

\NewDocumentCommand\fv{O{f}}{#1}

\newcommand*\fdual[1]{\fv_{#1}}

\newcommand*\rootcomp{1}

\newcommand*\func{\varphi}
\newcommand*\lfunc{\overline\func}

\newcommand*\tr{\triangleright}

\newcommand*\looproot{%
\begin{tikzpicture}[setree,]
	\placeroots{2}
	\jointrees{1}{1}
\end{tikzpicture}
}

\newcommand*\oneroot{
	\begin{tikzpicture}[setree]
		\placeroots{1}
	\end{tikzpicture}
}

\newcommand*\ffprime{
	\begin{tikzpicture}[setree]
		\placeroots{1}
		\children{child{node{}}}
	\end{tikzpicture}
	}

\usepackage{xspace}
\newcommand*\intmap{integrator map\xspace}
\newcommand*\Intmap{Integrator map\xspace}
\newcommand*\intmaps{integrator maps\xspace}
\newcommand*\Intmaps{Integrator maps\xspace}
\newcommand*\Bseriesmap{B-series map\xspace}
\newcommand*\Bseriesmaps{B-series maps\xspace}

\usepackage{authblk}

\newcommand*\email[1]{\href{mailto:#1}{\nolinkurl{#1}}}
\title{B-series methods are exactly the affine equivariant methods}
\author[1]{Robert I McLachlan\thanks{\email{r.mclachlan@massey.ac.nz}}}
\author[2]{Klas Modin\thanks{\email{klas.modin@chalmers.se}}}
\author[3]{Hans Munthe-Kaas\thanks{\email{hans.munthe-kaas@math.uib.no}}}
\affil[1]{Institute of Fundamental Sciences, Massey University, New Zealand} 
\affil[2]{Mathematical Sciences, Chalmers University of Technology, Sweden}
\affil[3]{Department of Mathematics, University of Bergen, Norway}
\author[4]{Olivier Verdier\thanks{\email{olivier.verdier@math.umu.se}}}
\affil[4]{Department of Mathematics and Mathematical Statistics, Ume\aa{} University, Sweden}

\begin{document}


\newcommand*\Man{\mathcal{M}}
\newcommand*\group[1]{\mathsf{#1}}
\NewDocumentCommand\Diff{O{\Man}}{\group{Diff}\IfNoValueTF{#1}{(\Man)}{(#1)}}
\newcommand*\dd{\mathrm{d}}
\newcommand*\supp{\operatorname{supp}}
\newcommand*\aff [2]{\operatorname{aff}(\RR^{#1},\RR^{#2})}

\newcommand*\Tvk{\spangle{T_k}}

\maketitle

\begin{abstract} 
	Butcher series, also called B-series, are a type of expansion, fundamental in the analysis of numerical integration. 
	Numerical methods that can be expanded in B-series are defined in all dimensions, so they correspond to \emph{sequences of maps}---one map for each dimension.
	A long-standing problem has been to characterise those sequences of maps that arise from B-series.
	This problem is solved here:
	we prove that a sequence of smooth maps between vector fields on affine  spaces has a B-series expansion if and only if it is \emph{affine equivariant}, meaning it respects all affine maps between affine spaces.
\end{abstract}

\begin{description}
	\item[Keywords] Butcher series · Equivariance · Aromatic series · Aromatic trees
	\item[Mathematics Subject Classification (2010)]  37C80 · 37C10 · 41A58
\end{description}


\section{Introduction} \label{sec:intro}

Let $\Phi(h,f)\colon \RR^{n} \to \RR^{n}$ be a numerical time-stepping method for the differential equation 
\begin{equation}\label{eq:plain_ode}
	\dot y = f(y), \quad y(0) = y_0\in\RR^{n}.
\end{equation}
That is, the time-stepping map $y_{k}\mapsto y_{k+1}$ is given by $y_{k+1} = \Phi(h,f)(y_{k})$, where $y_k \approx y(h k)$.
The convergence order of the method is obtained by comparing the Taylor expansion of $h\mapsto \Phi(h,f)(y_0)$ with the Taylor expansion of $h\mapsto y(h)$, using $\dot y = f(y)$ successively to avoid derivatives of~$y$.
If $\Phi(h,f)$ is a Runge--Kutta method, then the expansion is a linear combination of \emph{elementary differentials} of $f$.
For example, the first two terms for the midpoint method $y_{k+1} - y_{k} = h f(\frac{y_{k+1}+y_{k}}{2})$ are
\begin{equation}\label{eq:expansion_midpoint}
	y_{k+1} - y_{k} = h f(y_{k}) + \frac{h^{2}}{2} f'(y_{k})f(y_{k}) + \mathcal{O}(h^{3}).
\end{equation}
To work out higher order terms by direct methods is tedious and results in long, convoluted tables (see, for example, the work by \cite{Hu1956}).

In 1957, however, Merson~\cite{merson1957} rediscovered a remarkable structure, found already by Cayley~\cite{Ca1857} in 1857: a one-to-one correspondence between elementary differentials and \emph{rooted trees}.
This structure is the basis of the influential work by Butcher, who, in~1963, gave the first modern treatment of the order conditions for Runge--Kutta methods~\cite{butcher1963coefficients}, and, in~1974, developed an algebraic theory for series expansions of integration methods~\cite{butcher1972algebraic}.

Let $T$ denote the set of rooted trees.
The expansion of a Runge--Kutta method is of the form
\begin{equation}\label{eq:b_series_classical}
	y_{k+1} - y_{k} = \sum_{\tau \in T}h^{\abs{\tau}}\alpha(\tau)\eldfappl{\tau}{f}(y_{k}),
\end{equation}
where $\abs{\tau}$ denotes the number of vertices of~$\tau$, $\alpha\colon T\to\RR$ characterises the method, and $\eldfappl{\tau}{f}$ is the elementary differential of $f$ associated with~$\tau$ (see \autoref{sec:eldiff} for details).
The right-hand side of~\eqref{eq:b_series_classical} is called a \emph{Butcher series}, or \emph{B-series}, named so in 1974 by Hairer and Wanner~\cite{hairer1974butcher}.
The rich algebraic structure of B-series has since been studied extensively \cite{murua2006hopf,chartier2009algebraic,chartier2010algebraic,CaEFMa11,hairer2012conjugate}. 
A numerical integration method $\Phi(h,f)$ whose expansion in $h$ is of the form~\eqref{eq:b_series_classical} is called a \emph{B-series method}.
In addition to numerical contexts, B-series have arisen in other branches of mathematics, such as noncommutative geometry, in models of renormalization \cite{CoKr1999,brouder2000runge,Br2004} and rough paths theory \cite{Gu2010}.

Runge--Kutta methods are dense in the space of all B-series \cite[\S\,317]{Bu08}: given any series of the form~\eqref{eq:b_series_classical} and any $p\in\NN$, there exists a Runge--Kutta method whose B-series coincides up to order~$h^{p}$.
There are, however, methods $\Phi(h,f)$ other than Runge--Kutta whose expansions in $h$ are B-series.
Examples are Rosenbrock methods like $y_{n+1} = y_n + h (I-\frac{1}{2}h f'(y_n))^{-1}f(y_n)$ \cite{HaWa10}, exponential integrators like $y_{n+1} = y_n + h \varphi(h f'(y_n))f(y_n)$ where $\varphi(z) = (e^z-1)/z$, and the average vector field (AVF) method
$y_{n+1} = y_n + h \int_0^1 f(\xi y_{n+1} + (1-\xi)y_n)\, \dd \xi$ \cite{QuML08}.

So, which integration methods are B-series methods?
Of course, given some method $\Phi(h,f)$, one can always check~\eqref{eq:b_series_classical} by an expansion in~$h$.
But which properties \emph{characterise} B-series methods?
This is a natural, long-standing question that we answer here.
Our result is primarily based on two previous results: (i) that Runge--Kutta methods (and hence B-series methods) are \emph{equivariant} with respect to affine transformations~\cite{MLQu01}, and~(ii) that local, equivariant maps can be expanded in a type of series described by \emph{aromatic trees}~\cite{mk-verdier}.
Our result states that an integration method is a B-series method if and only if it defines an \emph{affine equivariant sequence}, meaning chiefly that it is equivariant and keeps decoupled systems decoupled.

Before going into the details of the result, we explain a few key points, fundamental throughout the paper.

\begin{itemize}
	\item
Many ODE integration methods (in particular B-series methods) fulfill 
\begin{equation}\label{eq:h_dep}
	\Phi(\varepsilon h,f) = \Phi(h,\varepsilon f)
\end{equation}
for any positive $\varepsilon$.
We may therefore disregard the dependency on $h$ and write $\Phi(h,f) = \Phi(h f)$.
\item
We regard an ODE integration method as a map $\Phi$ from a neighbourhood of zero of the smooth, compactly supported vector fields $\rvf{n}$ to the set of diffeomorphisms $\Diff[\RR^{n}]$ (see \autoref{def:integrator}).
To consider a neighbourhood of zero means to restrict the integration method to sufficiently small time steps.
Then $\Phi$ is an approximation of the \emph{exponential map} $\exp\colon\rvf{n}\to\Diff[\RR^{n}]$.
\item
We take the \emph{backward error analysis} point-of-view, which represents an integration method $\Phi(f)$ as $\Phi = \exp\circ\phi$ for some map $\phi \colon \rvf{n} \to \rvf{n}$. 
Here, equality holds in the sense of formal power series.
To avoid technical questions of convergence subordinate to our aim, the main result, \autoref{thm:main}, is formulated for maps from vector fields to vector fields, independently of their relation to numerical integration methods. 
The same argument is used in \cite[\S\,2.2]{mk-verdier}.

The key observation is, nevertheless, that $\Phi$ is a B-series method if and only if $\phi(f)$ can be expanded in a B-series, i.e.,
\begin{equation}\label{eq:b_series_bea}
	\phi(f) = \sum_{\tau\in T} \beta(\tau) \eldfappl{\tau}{f},
\end{equation}
for some map $\beta\colon T\to\RR$.
Each term in~\eqref{eq:b_series_bea} is a homogeneous polynomial in $f$, so each term corresponds to a symmetric, multi-linear map 
\begin{equation}\label{eq:bilinear_maps}
	\rvf{n}\times\cdots\times\rvf{n}\to \rvf{n}	
\end{equation}
evaluated at $f$.
For instance, the term $f' f$ in a B-series corresponds to the bilinear map $(f,g) \mapsto (f' g + g' f)/2$.
Consequently,~\eqref{eq:b_series_bea} is the Taylor series of $\phi\colon\rvf{n}\to\rvf{n}$, so our investigation consists of classifying those maps $\phi$ whose Taylor series are B-series.

\item
An ODE integration method actually corresponds to a \emph{sequence} of maps: one for each dimension $n\in\NN$.
From here on, we therefore use $\phi$ to denote a sequence of maps $\set{\phi_n}_{n\in\NN}$, where $\phi_n\colon\rvf{n}\to\rvf{n}$.
This  point of view is essential in the characterisation of B-series maps (see~\autoref{sec:main} for details).
\end{itemize}

The paper is organised as follows.
The main result is stated in \autoref{sec:main}.
In \autoref{sec:prelim} we give preliminary results necessary for the proof.
The main part of the proof is contained in \autoref{sec:mainproof}, and uses results from~\autoref{sec:special} on special vector field.
Finally, \autoref{sec:trans} connects the core result from \autoref{sec:mainproof} to the main result as stated in \autoref{sec:main}.

\section{Main result} \label{sec:main}

Our main result is a simple criterion to decide whether a method is a B-series method.
The essence of the result is captured as follows.

\begin{mdframed}[skipabove=10pt, skipbelow=10pt]
	 Let $\Phi = \{\Phi_n\}_{n\in\NN}$ be an integration method, defined for all vector fields on all dimensions $n$.
	Then $\Phi$ is a B-series method if and only if the property of
	\emph{affine equivariance} is fulfilled: if $a(x) \coloneqq A x + b$ is an affine map from $\RR^m$ to $\RR^n$, and $f\in\rvf{m}$, $g\in\rvf{n}$ fulfil $g(Ax + b) = A f(x)$, then $a \circ \Phi_m(f) = \Phi_n(g) \circ a$.
\end{mdframed}


The rigorous version of this result, using, as explained, the backward error analysis point of view, is stated in \autoref{thm:main} at the end of this section.
Before that,
we need to define the essential concept of equivariance.

%

Our definition of equivariance is an extended version of that in \cite[\S\,2.4]{mk-verdier}, \cite[\S\,4.3]{MLQu01}.
We first define the main object of study: sequences of smooth maps from a neighbourhood of zero in the space of compactly supported vector fields to itself.
As we reuse this object several times, we encapsulate it in the following definition. 
\begin{definition}[\Intmap]
	\label{def:integrator}
	Let $\rvf{n}$ be the space of compactly supported vector fields on $\RR^n$ with the test function topology.
	We define an \emph{\intmap} as a sequence of smooth maps
\begin{equation}\label{eq:family_maps}
	\phi = \set[\big]{\phi_n\colon U_n\subset\rvf{n}\rightarrow\rvf{n}}_{n\in \NN},
\end{equation}
where each $U_n$ is a suitable neighbourhood of the zero vector field in $\rvf{n}$.
\end{definition}
In order for an \intmap $\phi$ to correspond to a B-series, there must clearly be some relationship between the individual maps $\phi_n$.
We denote by $\aff{n}{m}$ the set of affine maps:
\begin{equation}
\aff {n}{m} = \setc[\big]{a\colon \RR^n\rightarrow \RR^m}{a(x) = Ax+b, \text{$A\in \RR^{m\times n}$, $x\in \RR^n$ and $b\in \RR^m$}}
.
\end{equation}
Pullback of vector fields along invertible diffeomorphisms is generalised for non-invertible maps by the concept of \emph{intertwining} (relatedness) of vector fields. 
We say that $a\in \aff{n}{m}$ \emph{intertwines} the vector fields $f\in \rvf{n}$ and $g\in \rvf{m}$, which we denote by  $f\affrel g$, if
\begin{align}
	\label{eq:defrelated}
	g(Ax+b) = Af(x)
	.
\end{align} 

\begin{definition}[Affine equivariance]\label{def:sae} 
	An \intmap~$\phi$ (\autoref{def:integrator}) is called \emph{affine equivariant} if, for all $m,n\in \NN$ and all $a\in \aff {m}{n}$,
\begin{align}
	f\affrel g \quad \implies \quad \phi_m(f)\affrel\phi_n(g) .
\end{align}
\end{definition}

Here are some of the properties attributed to affine equivariance:
\begin{enumerate}
\item
\Intmaps that are equivariant with respect to \emph{invertible} affine maps preserve affine symmetries.
At every fixed dimension, the integration method corresponding to such an \intmap is affine equivariant in the standard sense, used in \cite{mk-verdier}.
\item
Consider systems of the form
\begin{align}
\dot x &= f(x), \\
\dot y &= g(x,y). 
\end{align}
where $x\in\RR^m$, $y\in\RR^n$.
Geometrically, they are characterised by preservation of the foliation by planes $x=$ constant \cite{mclachlan2001kinds}.
\Intmaps $\phi$ (and corresponding integration methods) that are equivariant with respect to affine \emph{surjections}, in this case the surjection $(x,y)\mapsto x$, preserve this foliation.
In addition, they are ``closed with respect to closed subsystems'' \cite{bochev1994quadratic}: the map $\phi_{m+n}((f,g))$  restricted to $x$ in domain and range is identical to $\phi_m(f)$. 
\item 
\Intmaps (and corresponding integration methods) that are equivariant with respect to affine \emph{injections} preserve affine weak integrals.
(Recall that a weak integral is a function $I\colon \RR^n\to\RR$ such that $I=0$ implies $\dot I=0$; an affine weak integral is equivalent to an invariant affine subspace.)
In addition, the map on the affine subspace induced by such an \intmap is identical to that produced by application directly to the system on the subspace.
\end{enumerate}
Each of these are also properties of B-series methods, while non-B-series methods, such as partitioned Runge--Kutta methods, do not have them.
In addition, B-series methods have many other structural and geometric properties \cite{IsQuTs07,chartier2006algebraic}.
Perhaps surprisingly, all of these are now seen to be consequences of affine equivariance.

To give a rigorous definition of \intmaps corresponding to B-series methods, we need Taylor series of vector field maps.
Let $\phi_{n}\colon\rvf{n}\to\rvf{n}$ be smooth.
Its $k$:th derivative at $0\in\rvf{n}$, denoted $D^k\phi_{n}(0)$, is a $k$--linear, symmetric form on $\rvf{n}$.
Taylor's formula~\cite[Theorem~I.5.12]{KrMi97} states that
\begin{multline}\label{eq:taylor}
	\phi_n(f) = \phi_n(0) + \frac{D\phi_n(0)[f]}{1!} + \cdots + \frac{D^{k}\phi_n(0)[\overbrace{f,\ldots,f}^{k}]}{k!} \\
	+ \int_0^1 \frac{(1-\sigma)^{k}}{k!}D^{k+1}\phi(\sigma f)[\overbrace{f,\ldots,f}^{k+1}]\,\dd \sigma .
\end{multline}

Let, as before, $T$ denote the set of rooted trees and $\eldfappl{\tau}{f}\in\rvf{n}$ denote the elementary differential of $f\in\rvf{n}$ associated with $\tau\in T$.
Further, let $\Tvk$ denote the free $\RR$-vector space over the set $\Tree_k$ of trees with $k$ vertices.
That is, each element in $\Tvk$ is an $\RR$-linear combination of elements in $\Tree_k$.
By construction, $\Tree_k$ is a basis for~$\Tvk$.

For each $k\in\NN$ and $f\in\rvf{n}$, $\eldfappl{\cdot}{f}$ is naturally extended to a linear map $\Tvk\to \rvf{n}$.
We define \Bseriesmaps as those \intmap whose Taylor coefficients are elementary differentials.

\begin{definition}[\Bseriesmap]
	\label{def:bseriesintegrator}
	An \intmap $\phi$ (\autoref{def:integrator}) is called a \emph{\Bseriesmap} if, for each $k\in\NN$, there exists $\tau_k\in\Tvk$ such that for all $n\in\NN$
	\begin{align}
		\frac{D^{k}\phi_n(0)[f,\ldots,f]}{k!} = \eldfappl{\tau_k}{f}, \qquad \forall\, f\in\rvf{n}.
	\end{align}
\end{definition}

We are now in a position to state the main result of this paper.
\begin{theorem} \label{thm:main}
	Let $\phi$ be an \intmap (\autoref{def:integrator}).
	Then $\phi$ is a \Bseriesmap (\autoref{def:bseriesintegrator}) if and only if it is  affine equivariant (\autoref{def:sae}).
\end{theorem}

\begin{proof}
	Using \autoref{pro:transfer} and \autoref{prop:transferpoint}, this result is equivalent to \autoref{prop:bijection}.
\end{proof}

\begin{remark}
	It turns out that, although B-series are affine equivariant, it is not necessary to use equivariance with respect to all affine maps to prove \autoref{thm:main}.
	We only use equivariance with respect to all \emph{surjective} affine maps to prove \autoref{prop:bijection} (invertible affine maps in \autoref{prop:surjapoly} and surjective affine maps in \autoref{prop:blockequi}) and with respect to the trivial affine injections from $\RR^0$ to $\RR^n$ (in \autoref{prop:equilocal}).
\end{remark}

\subsection{Idea of the proof}

The proof of \autoref{thm:main} is long and contains many details.
In this section we therefore explain its main ingredients by proving \autoref{thm:main} in the special case when the functions $\phi_m$ are \emph{homogeneous quadratic} maps, instead of arbitrary nonlinear smooth maps.

We use a transfer argument, similar to that of \cite{mk-verdier}, to transfer the statement from affine equivariant \intmaps, to linear equivariant polynomials $P_m$ in the derivatives of the vector fields (this is \autoref{pro:transfer} and \autoref{prop:transferpoint}).

After the transfer argument, we  have for every dimension $m$ a vector valued homogeneous quadratic polynomial $P_m$ in the derivatives of the vector field at zero (cf. \autoref{sec:forms}).
A component of such a vector could be for  $ f_2^3 f_1^1$  in dimension three, or a linear combination of such expressions (we follow standard practice for the notation of components and partial derivatives, which are detailed in \autoref{sec:prelim}).

We make explicit the assumption that the sequence $P_m$ is linear equivariant.
Consider an arbitrary linear map $A$ from $\RR^n$ to $\RR^m$.
We say that $A$ intertwines a vector field $f$ in dimension $n$ and a vector field $g$ in dimension $m$, if $g(A x) = A f(x)$ for any $x$ in $\RR^n$ (cf. \autoref{sec:linintertwining}).
The sequence $\set{P_m}_{m\in\NN}$ is equivariant if whenever $A$ intertwines $f$ and $g$, it also intertwines $\dappl{P_n}{f}$ and $\dappl{P_m}{g}$ (cf. \autoref{sec:equiseq}).

We now prove the special case of quadratic linear equivariant functions.

\begin{proposition}
	Suppose that $P = \set{P_m}_{m\in\NN}$ is a  vector-valued, homogeneous quadratic polynomial in the derivatives of vector fields at zero.
	Suppose further that for any linear map $A \colon \RR^n \to \RR^m$ and vector fields such that $g(Ax) = A f(x)$, we have $\dappl{P}{g}(Ax) = A \dappl{P}{f}(x)$.
	Then $P$ must be of the form
	\begin{equation}
	\dappl{P_m}{f} =  \lambda \sum_{i,j=1}^m f_i^j f^i \partial_j
	\end{equation}
	for some scalar $\lambda$.
\end{proposition}

\newcommand*\combi[1]{\lambda_{#1} \ffprime + \mu_{#1} \looproot}
\NewDocumentCommand\Proj{O{m}O{2}}{\Pi^{#1}_{#2}}


We use the following standard notations (see \autoref{sec:eldiff} for the notations pertaining to elementary differentials used in this paper):
\begin{gather}
	\label{eq:eldifftwo}
	\begin{aligned}
		\eldappl{\ffprime}{f} &= \sum_{i,j=1}^n f_j^i f^j \partial_i
		,
		\\
		\eldappl{\looproot}{f} &= \sum_{i,j=1}^n f_i^i f^j \partial_j
		.
	\end{aligned}
\end{gather}
The main ingredient of the proof is to notice that the term $\looproot$ is \emph{coupling} (cf. \autoref{prop:blockequi} and \autoref{le:parteldiff}), which contradicts equivariance (cf. \autoref{le:tpr}).

\begin{proof}

	Consider an arbitrary vector field $f$ in dimension $k$ (in this proof, $k$ will take the value~$1$ or~$2$).
For any dimension $m \geq k$, we denote the vector $f \oplus 0$ as the vector equal to $f$ on the first $k$ components, padded with zeros (cf. \autoref{sec:partitionedvf}), i.e.,
\begin{equation}\paren{f\oplus 0} (x_1,\ldots,x_k,\ldots,x_m) \coloneqq (f^1(x_1,\ldots,x_k),\ldots, f^k(x_1,\ldots,x_k),0,\ldots,0)
	.
\end{equation}
For $m \geq k$, we also define the linear projection $\Proj[m][k] \colon \RR^m \to \RR^k$ which keeps the first $k$ components.
Note that the projection $\Proj[m][k]$ intertwines $f\oplus 0$ and $f$, since $\Proj[m][k] \paren{f\oplus 0} (x) = f(\Proj[m][k] x)$ (cf. \autoref{prop:defpartition}).
Finally, observe that from the formulas \eqref{eq:eldifftwo}, we have for any scalars $\lambda$ and $\mu$ (cf. \autoref{prop:eldiffcompat})
\begin{equation}
	\label{eq:equizero}
	\Proj[m][k]\dappl{\paren{\combi{}}}{f \oplus 0} = {\dappl{\paren{\combi{}}}{f}} 
.
\end{equation}
Given these definitions and observations, the proof proceeds as follows.

	\begin{enumerate}
		\item
Using equivariance with respect to \emph{invertible} linear maps on each dimension,  we obtain from \cite{mk-verdier} (see \autoref{prop:surjapoly}) that, for every dimension $m$, we have
\begin{equation}
	\dappl{P_m}{f} = \eldappl{\paren{\combi{m}}}{f}
\end{equation}
for some real numbers $\lambda_m$ and $\mu_m$.

\item
We now proceed to show that $\dappl{P_m}{f} = \dappl{\paren{\combi{2}}}{f}$ for any vector field $f$ in $m$ dimensions (cf. \autoref{prop:Asurjective}).
We first show that for $m \geq 2$ we have $\lambda_m = \lambda_2$ and $\mu_m = \mu_2$, and then show that, on one-dimensional vector fields, $\dappl{P_1}{f} = \dappl{\paren{\combi{2}}}{f}$.
	\begin{enumerate}
\item

Consider a vector field $f$ in dimension two.
As $\Proj$ intertwines $f\oplus 0$ and $f$, equivariance of $P$ leads to $\Proj \dappl{P_m}{f \oplus 0}(x) =  \dappl{P_2}{f}(\Proj x) = \eldappl{\paren{\combi{2}}}{f}(\Proj x)$.
Combining this with \eqref{eq:equizero}, we obtain that  $\combi{m}$ and $\combi{2}$ coincide on vector fields in dimension two.
Appealing again to \cite{mk-verdier} (see \autoref{prop:surjapoly}), we obtain that $\lambda_m = \lambda_2$ and $\mu_m = \mu_2$.

\item
	Consider now a vector field $f$ in  one dimension.
	Construct $f\oplus 0$ in two dimensions; equivariance now gives $\Proj[2][1] \dappl{P_2}{f\oplus 0}(x) = \dappl{P_1}{f}(\Proj[2][1] x)$.
Again, by comparing this with \eqref{eq:equizero}, we obtain that $\dappl{P_1}{f} = \dappl{\paren{\combi{2}}}{f}$.

\end{enumerate}
	Combining the last two items, we deduce that in fact $\dappl{P_m}{f} = \eldappl{\paren{\combi{2}}}{f}$ for any integer $m$.

\item
The final step is now to show that $\mu_2 = 0$ (cf. \autoref{prop:Bsurjective}).
The idea is to apply $P_2$ to the special vector field
$g(x_1,x_2) = (1,x_2)$ (cf. \autoref{prop:monomial} for the general expression of such special vector fields).
One checks that, at $x = 0$, we have $\eldappl{{\ffprime}}{g}(0) = 0$, and that $\eldappl{{\looproot}}{g}(0) = (1,0)$.
We conclude that \(\dappl{P_2}{f}(0) = \mu_2 (1,0)\), and in particular $\Proj[2][1]{\dappl{P_2}{f}(0)} = \mu_2$.
Denote by $1$ the constant vector field on $\RR$; it is clear that $\Proj[2][1]$ intertwines $g$ and $1$.
The equivariance of $P$ now implies that $\Proj[2][1]\dappl{P_2}{g}(0) = \dappl{P_1}{1}(0) =  0$, which entails $\mu_2 = 0$.
\end{enumerate}
We conclude that for any integer $m$, $P_m = \lambda_2 \ffprime$, which, recalling the notation \eqref{eq:eldifftwo}, is the claim of the proposition.
\end{proof}

The rest of the paper consists of generalizing the arguments in the proof above, in order to accommodate  homogeneous polynomials of any degree.

\section{Preliminary definitions}\label{sec:prelim}

\subsection{Polynomial vector fields}

For a fixed dimension $n$, we define an infinite dimensional vector space $\psf$ of \demph{polynomials} of arbitrary degree.
We use derivatives as coordinates in that space.
These coordinates are thus indexed by the partial derivatives, as
\begin{align}
(\fv, \fv_1,\ldots,\fv_n,\fv_{11},\fv_{12},\dots)
\end{align}
with appropriate symmetry conditions, such as $\fv_{12} = \fv_{21}$.

We denote by $\pvf[n]$ the set of \demph{vector-valued polynomials}, which consists of $n$ elements of $\psf$, that is $\pvf[n] = (\psf)^n$.
An element in $\pvf[n]$ should be regarded as a \demph{polynomial vector field}.

\subsection{Forms}
\label{sec:forms}

A \demph{$k$-form} in dimension $n$ is a homogeneous polynomial of degree $k$ on the space of polynomial vector fields $\pvf[n]$.
We denote scalar $k$-forms on $\pvf$ by $\sym^k(\pvf)$.
A \demph{vector valued $k$-form} in dimension $n$ is a list of $n$ $k$-forms, regarded as a vector.
It is thus an element of $\RR^n \otimes \sym^k(\pvf)$.
As is customary, we use the basis~$\partial_i$.
Note that we use the same notation for the basis in all dimensions, which should not cause confusion.
For instance, in dimension $n$ the map $\eta$ defined by
\begin{align}\dappl{\eta}{\fv}
		= \fv^1 \fv^1_1 \partial_1 + \fv^n_1 \fv^n_n  \partial_n 
\end{align} 
is a vector-valued $2$-form.
In one dimension, the coordinates are $(f^1, f^1_1, f^1_{11},\dots)$, corresponding to $(f,f',f'',\ldots)$, and a $2$-form in one dimension could be $f f' + (f'')^2$. 
In two dimensions, an example of a vector valued $3$-form is $\dappl{P}{f} = f_{112}^2 f^2 f^1 \partial_1 + (f_2^1)^3 \partial_2$.

\subsection{Intertwining}
\label{sec:linintertwining}

Given a linear map $A \in \mathcal{L}(\RR^n, \RR^m)$ we say that $A$ \demph{intertwines} $f\in \pvf[n]$ with $g\in \pvf[m]$, denoted \begin{align}f \affrel[A] g
,
\end{align} if the equality
\begin{align}
	\label{eq:deflinintertwine}
	{g}(A x) = A {f}(x)
\end{align}
is valid for all $x\in\RR^{n}$.

We give an example of intertwining with respect to a projection. 
Define the scalar polynomials $f_1 \in \psf[1]$ and $f_2 \in \psf[2]$,
 the polynomial vector field $f\in\pvf[2]$ by $f(x_1,x_2) \coloneqq f_1(x_1) \partial_1 + f_2(x_1,x_2) \partial_2$,
 and $g\in\pvf[1]$ by $g(x) = f_1(x) \partial_1$.
We denote the projection $\pi \in \Lin(\RR^2,\RR^1)$ on the first coordinate, that is, $\pi(x_1,x_2) = x_1$.
In that case, one can check that $g\paren[\big]{\pi(x)}= \pi \paren[\big]{f (x)}$, so we have $f \affrel[\pi] g$.

We now give an example of intertwining with respect to an injection. 
Define the vector field $f\in\pvf[1]$ by $f = f_1(x_1) \partial_1$, where $f_1 \in \psf[1]$.
Consider now the vector field $g\in\pvf[2]$ defined by $g(x_1,x_2) = g_1(x_1,x_2) \partial_1 + g_2(x_1,x_2)\partial_2$, with the property that $g_2(x,0) = 0$ and $g_1(x,0) = f_1(x)$ for any $x\in\RR$.
Define the injection $i \in\Lin(\RR,\RR^2)$ by $i(x) = (x,0)$.
As $g\paren[\big]{i(x)} = i\paren[\big]{f(x)}$, we have $f\affrel[i] g$.

\subsection{Equivariant sequences}
\label{sec:equiseq}

We define an \demph{equivariant sequence of (vector-valued) $k$-forms} as a sequence $\eta_n$ of vector-valued $k$-forms with the property
\begin{align}
	\label{eq:deflinequi}
f \affrel[A] g,
\end{align}
it holds that
\begin{align}
	\label{eq:deflinequi2}
\eta_n(f) \affrel[A] \eta_m(g).
\end{align}
A typical example of an equivariant sequence of $2$-forms is
\begin{align}
	\eta_n[f] \coloneqq  \sum_{i,k=1}^n \fv_i^k f^i \partial_k.
\end{align}

Finally, we use the following simplifying notation.
If $\fv\in\pvf[n]$ and $P$ is a sequence of $k$-forms, we use the notation
\begin{align}
	\dappl{P}{\fv} \coloneqq \dappl{P_n}{\fv}
	.
\end{align}
If the dimension $n$ is not clear from the context, we use the explicit form $\dappl{P_n}{\fv}$.


\subsection{Aromatic forests, trees and molecules}

We now review some definitions from \cite{mk-verdier}. 
Let $\Graph$ denote the set of all directed graphs with a finite number of vertices, where each vertex has zero or one outgoing edges. 
A vertex with no outgoing edges is called a \demph{root}. 
For $\gamma\in\Graph$, let $\Vertices(\gamma)$ and $\Edges(\gamma)$ denote the vertices and edges of the graph, let $\Roots(\gamma)\subset \Vertices(\gamma)$ denote the \demph{root vertices}. 
For $v\in \Vertices(\gamma)$, let $\Parents(v)\subset \Vertices(\gamma)$ denote the set of \demph{parent vertices}, $\Parents(v) := \setc{p\in \Vertices(\gamma)}{(p,v)\in \Edges(\gamma)}$. 
Let $\abs{\gamma} \coloneqq \#{\Vertices(\gamma)}$ denote the number of vertices and $\abs{\Roots(\gamma)}$ the number of roots.

We have $\Graph = \bigcup_{r=0,k=1}^\infty\Graph[r][k]$, where $\Graph[r][k]$ denotes graphs with $r$ roots and $k$ vertices.
We denote $\Graph[r]:=\bigcup_{k=1}^\infty \Graph[r][k]$.
Let $\Gv$ and $\Gv[r]$ denote the free $\RR$-vector spaces over $\Graph$ and $\Graph[r]$.
An element of  $\Graph[1]$ is called an \demph{aromatic tree}. 
For instance, the following is an element of $\Graph[1]$, as it has one root, so it is an aromatic tree:
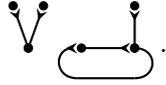
\newcommand*\exampleatree{
		\placeroots{3}
		\children[1]{child{node{}} child{node{}}}
		\children[3]{child{node{}}}
		\jointrees{2}{3}
}
\begin{equation}
	\begin{tikzpicture}[etree]
		\exampleatree
\end{tikzpicture}.
\end{equation}
By convention, we will assume that the cycles are always oriented counterclockwise and we will draw the aromatic trees in short form as:
\begin{align}
	\begin{tikzpicture}[setree]
		\exampleatree
	\end{tikzpicture}.
\end{align}

The set of \demph{trees} is the subset $T\subset \Graph[1]$ of connected graphs in $\Graph[1]$. 
Similarly, the set of \demph{aromatic molecules} is the subset $M\subset \Graph[0]$ of connected graphs in $\Graph[0]$.

We define the \demph{product of graphs} as their disjoint union: for $\gamma_1,\gamma_2\in \Graph$ the product $\gamma_1\gamma_2=\gamma_2\gamma_1$ is the graph consisting of the union of the vertices and edges of the two graphs.

\begin{lemma}\label{lem:product_decomposition}
	Let $\gamma\in \Graph[1]$.
	Then $\gamma$ can be decomposed as
	\begin{equation}
		\gamma = \mu_1\mu_2\cdots \mu_k\tau,
	\end{equation}
	where $\mu_1,\ldots,\mu_k\in M$ and $\tau\in T$.
\end{lemma}
\begin{proof}
Any graph $\gamma\in \Graph[1]$ can be decomposed  into a union of its connected components, where each connected component is either in $T$ or in $\Tc$. 
As $\gamma$ has one root, the root must belong to one of the components, which is thus in $\Tree$.	
The other components are also aromatic forests, but with the same number of nodes and arrows, so they cannot have any root and must belong to $\Tc$.
\end{proof}

\subsection{Elementary differentials}
\label{sec:eldiff}

Consider $f\in \pvf[n]$ and $\gamma \in \Graph[r][k]$.
The set $\Vertices(\gamma)$ denotes the set of vertices of $\gamma$.
For a node $i \in \Vertices(\gamma)$, we denote by $\Parents(i)$ the parent vertices of $i$.
We define the \demph{elementary differential} $\eldfappl{\gamma}{f}[n]\in (\RR^n)^{\otimes r}\otimes\sym^k(\pvf[n])$ in tensor component notation as
\begin{equation}
	\label{eq:eldiff2}
\eldfappl{\gamma}{\fv}[n] \coloneqq \prod_{i\in \Vertices(\gamma)} f^i_{\Parents(i)} 
,
\end{equation}
where we use the \emph{Einstein summation convention}: repeated indexes are summed over in the range $\set{1\ldots n}$. 
Every lower index is paired with an upper index, but the upper indices, corresponding to roots, are not paired. 
We rewrite that definition in a more tractable form in \eqref{eq:eldiff2b}.
Note that we make sure to keep track of the \emph{dimension} $n$ in those expressions.
Consider, for instance, the expression
\begin{align}
\eldiff_n\paren[\big]{\looproot}[f] =
 \sum_{i=1}^n f^i\partial_i \sum_{j=1}^{n} f^j_j .
\end{align}
In one dimension it is
\begin{align}
	\eldiff_1\paren[\big]{\looproot}[f] = f^1 f^1_1 \partial_1 \equiv f f' .
\end{align}
In two dimensions it is
\begin{align}
	\eldiff_2\paren[\big]{\looproot}[f] = \paren{f^1 \partial_1 + f^2 \partial_2} (f_1^1 + f^2_2) = f \operatorname{div}f.
\end{align}

Below is an example of an elementary differential where $\gamma\in \Tc$ and $\eldfappl{\gamma}{\fv}[n]$ is a scalar two-form:
\begin{align}
	\eldfappl{%
\begin{tikzpicture}[setree]
	\placeroots{1}
	\children[1]{child{node{}}}
	\jointrees{1}{1}
\end{tikzpicture}%
}{f}[n] = f^i f^j_{ij} = \sum_{i=1}^n\sum_{j=1}^n f^i\frac{\partial^2 f^j}{\partial x_i\partial x_j}  = f\cdot \operatorname{grad}(\operatorname{div} f) .
\end{align}
Here, $i$ is the the top vertex of $\gamma$, giving the part $f^i$ (no parents, hence no subscript), and $j$ is the bottom vertex of $\gamma$, giving the factor $f^j_{ij}$, since $\Parents(j) = \{i,j\}$.



We now rewrite the definition~(\refeq{eq:eldiff2}) of the elementary differential $\eldfappl{\gamma}{\cdot}[n]\colon \pvf[n]\rightarrow\pvf[n]$ in an equivalent, but, for our purpose,  more tractable form:
\begin{equation}
	\label{eq:eldiff2b}
\eldfappl{\gamma}{\fv}[n] = \sum_{\nu\colon \Vertices(\gamma)\rightarrow[n]}\prod_{v\in \Vertices(\gamma)} f^{\nu(v)}_{\nu\paren{\Parents(v)}}\prod_{r\in \Roots(\gamma)}\partial_{\nu(r)},
\end{equation}
Here, $[n] \coloneqq \set{1,2,\ldots,n}$, $\Roots(\gamma)$ denotes the root vertices of the graph $\gamma$, $\partial_j\in \RR^{n}$ denotes the unit vector in a direction $j\in [n]$, and $\Parents(v)$ denotes the set of parent vertices of vertex $v$, i.e.,
$\Parents(v) = \setc{p\in \Vertices(\gamma)}{(p,v)\in \Edges(\gamma)} $.
The sum runs over all possible maps $\nu\colon \Vertices(\gamma)\rightarrow [n]$, assigning vertices in $\gamma$ to integers in $[n]$.

To verify the re-writing~\eqref{eq:eldiff2b}, we note that an assignment of vertices in $\gamma$ to 
integers in $[n]$ is already implicit in our interpretation of~(\refeq{eq:eldiff2}) where $i$ both denotes a node in $\gamma$ and an integer in $[n]$. Thus $f^i_J$ in~(\refeq{eq:eldiff2}) is the same as $f^{\nu(v)}_{\nu\paren{\Parents(v)}}$
in~(\refeq{eq:eldiff2b}). In~(\refeq{eq:eldiff2}), the root nodes are left as tensor components which are not summed over. In~(\refeq{eq:eldiff2b}) we pair the root nodes $\nu(r)$ with
the basis vector field $\partial_{\nu(r)}$, and hence the sum here runs over all possible maps $\nu\colon \Vertices(\gamma)\rightarrow[n]$.


For each dimension $n$, equation~\eqref{eq:eldiff2b} defines the \demph{elementary differential map} $\eldiff_n$.
In particular, on the subspace $\spangle{\Graph[1][k]}$ of aromatic trees with $k$ vertices, we have
\begin{align}
	\eldiff_n \colon \spangle{\Graph[1][k]} \to \RR^n \otimes \sym^k(\pvf[n])
	.
\end{align}

For fixed dimension $n$, the map $\eldiff_n$ is not injective.
For instance,
\begin{align}
	\eldfappl%
{
	\begin{tikzpicture}[setree,]
		\placeroots{3}
		\jointrees{2}{2}
		\jointrees{3}{3}
	\end{tikzpicture}
	+ 2\ 
	\begin{tikzpicture}[setree,]
		\placeroots{1}
		\children[1]{child{node{} child{node{}}}}
	\end{tikzpicture}
	-2\ 
	\begin{tikzpicture}[setree,]
		\placeroots{2}
		\children[1]{child{node{} }}
		\jointrees{2}{2}
	\end{tikzpicture}
	-
	\begin{tikzpicture}[setree,]
		\placeroots{3}
		\jointrees{2}{3}
	\end{tikzpicture}
}{f}[2]
	= 0 
	\qquad\text{for all $f\in\pvf[2]$.}
\end{align}
The elementary differential map is not injective even when restricted to $\Tv$.
For instance,
\begin{align}
\eldfappl%
{\begin{tikzpicture}[setree,]
		\placeroots{1}
		\children[1]{child{node{}} child{node{}child{node{}}}}
	\end{tikzpicture} -
	  \ {\begin{tikzpicture}[setree,]
		\placeroots{1}
		\children[1]{child{node{}child{node{}}child{node{}}}}
	\end{tikzpicture}}}{f}[1]=0
	\qquad \text{for all $f\in\pvf[1]$.}
\end{align}
This is one of the motivations for regarding the elementary differential as acting on all dimensions.
Indeed, we build a \emph{sequence}
\begin{align}
	\eldiff \colon \spangle{\Graph[1][k]} \to \set{\text{sequences of vector-valued $k$-forms}}
\end{align}
defined by
\begin{align}
	\eldiff(\gamma) \coloneqq \paren[\big]{\eldiff_1(\gamma), \eldiff_2(\gamma),\ldots}.
\end{align}

In the sequel, we use the following simplified notation:
for $\gamma \in \Graph$ and $\fv\in\pvf[n]$, we define
\begin{align}
	\eldappl{\gamma}{\fv}[n] \coloneqq \eldfappl{\gamma}{f}[n] .
\end{align}
Note that the dimension $n$ is implicitly defined by the space $\pvf[n]$ that $\fv$ belongs to.
When ambiguity remains, we resume the explicit notation $\eldfappl{\gamma}{f}[n]$.

As a first result for the elementary differential map, consider the following result, established in the paper \cite[\S\,7.4]{mk-verdier}:

\begin{proposition}
	\label{prop:surjapoly}
	For each dimension $n\in\NN$ and degree $k\in\NN$, 
$\eldiff_n$ is a surjection from $\spangle{\Graph[1][k]}$ to the space $\RR^n \otimes \sym^k(\pvf[n])$ of vector-valued equivariant $k$-forms in dimension $n$.
Moreover, if $k \leq n$, then $\eldiff_n$ is a bijection.
\end{proposition}


\section{Special vector fields}\label{sec:special}

The proof of \autoref{prop:bijection} below is based on the construction of special vector fields.
In particular, we need vector fields with block-diagonal Jacobians.
We also need special vector fields that form a dual basis with respect to aromatic trees and molecules.

\subsection{Partitioned vector fields}
\label{sec:partitionedvf}

Vector fields with block-diagonal Jacobian matrices are called \demph{partitioned}.
They serve a special role in the sequel. 

\begin{proposition}
	\label{prop:defpartition}
	Consider the decomposition
	\begin{align}
		\RR^{n+m} = \RR^n \oplus \RR^m,
	\end{align}
	and denote the associated projections by $\pi_1$ and $\pi_2$.
	Then, given $\fv \in \pvf[n]$ and $\fv[g] \in \pvf[m]$, there is a unique vector field $\fv[h] \in \pvf[n+m]$ characterized by
	\begin{align}
		\fv[h] \affrel[\pi_1] \fv[f]
		,
		\qquad
		\fv[h] \affrel[\pi_2] \fv[g]
		.
	\end{align}
	That vector field is denoted
	\begin{align}
		\fv \oplus \fv[g] \coloneqq \fv[h]
	.
\end{align}
\end{proposition}
\begin{proof}
$\fv [h]\affrel[\pi_1] \fv	$ means by definition \eqref{eq:deflinintertwine} that $\pi_1 \fv[h](x,y) = \fv(x)$, where we denote a point $(x,y) \in \RR^{n+m}$ so that $\pi_1(x,y) = x$.
This means that $\fv[h](x,y)$ has $f(x)$ as first components, and likewise, $g(y)$ as last components.
\end{proof}

We thus immediately obtain the following property of equivariant sequences.

\begin{lemma}
	\label{prop:blockequi}
	If 
	$\fv[f]\in \pvf[n]$ and $\fv[g] \in \pvf[m]$, and $P$ is an equivariant sequence of $k$-forms, we have
	\begin{align}
		\dappl{P}{\fv[f \oplus g]} = \dappl{P}{\fv[f]} \oplus \dappl{P}{\fv[g]}
		.
	\end{align}
\end{lemma}
\begin{proof}
	Consider the vector field $h = f \oplus g$, defined as in \autoref{prop:defpartition}.
	By definition of the equivariance of the sequence \eqref{eq:deflinequi}, we have $\dappl{P}{h} \affrel[\pi_1] \dappl{P}{f}$, and $\dappl{P}{h} \affrel[\pi_2] \dappl{P}{g}$, which, again by \autoref{prop:defpartition} ensures that $\dappl{P}{h} = \dappl{P}{f} \oplus \dappl{P}{g}$.
\end{proof}

Thus, equivariant sequences ``keep decoupled systems decoupled'';
this is essentially the difference between aromatic series and B-series.

We now derive special formulas for elementary differentials of partitioned vector fields.
To do that, we first reformulate the elementary differential formula \eqref{eq:eldiff2b} using \emph{dependency graphs}.

\begin{definition}
	\label{def:depgraph}
	The \emph{dependency graph} of $f\in\pvf[n]$, denoted $\dgr(f)$, is the directed, labeled graph defined by
	\begin{subequations}
	\begin{align}
	\Vertices(\dgr(f)) &= [n]\\
	(j,i) \notin \Edges(\dgr(f)) &\iff \partial_{j} f^i(x) = 0 \quad\forall \, x\in\RR^{n}
	.
	\end{align}
\end{subequations}
\end{definition}
Note that a vector field is partitioned if and only if its dependency graph is disconnected.

\begin{lemma}\label{le:eldiff}For $\mu\in M$ the elementary differential is given as
\begin{equation}\label{eq:eldiff3}
\mu[f] = \sum_{\nu\in \hom(\mu,\dgr(f))} \prod_{v\in \Vertices(\mu)}f^{\nu(v)}_{\nu(\Parents(v))}, 
\end{equation}
where $\hom(\mu,\dgr(f))$ denotes graph homomorphisms of $\mu$ into $\dgr(f)$, i.e.\ a map of graphs sending vertices to vertices and edges to edges.
For $\tau\in T$, component $k$ of the elementary differential is given as
\begin{equation}\label{eq:eldiff4}
\tau[f]^k= \sum_{\stackrel{\nu\in \hom(\tau,\dgr(f))}{\nu(\Roots(\tau))=k}}  \prod_{v\in \Vertices(\mu)}f^{\nu(v)}_{\nu(\Parents(v))}.
\end{equation}
\end{lemma}

\begin{proof}
Equation~(\refeq{eq:eldiff2b}) expresses the elementary differential as a sum over all possible maps $\nu\colon \Vertices(\gamma)\rightarrow \Vertices\paren[\big]{\dgr(f)}$. 
By definition of the dependency graph \autoref{def:depgraph}, we see that all maps which are not sending edges to edges must yield 0. 
Hence, we can restrict the sum to homomorphisms. Both formulas follow from this argument; in the
first case $\mu$ has no roots. In the latter case component $k$ is the multiplier in front of $\partial_{\nu\paren[\big]{\Roots(\tau)}}$, i.e., $k=\nu\paren[\big]{\Roots(\tau)}$. Thus we restrict
to all homomorphisms sending the root of $\tau$ to $k$.
 \end{proof}


Our next result shows that trees and molecules preserve, in a sense, the structure of partitioned vector fields.

\begin{lemma}\label{le:parteldiff}
Consider the partitioned vector field $f = f_1\oplus f_2 \in \pvf[m+n]$, with $f_1\in\pvf[m]$ and $f_2\in\pvf[n]$. 
If $\mu\in M$ and $\tau\in T$, then
\begin{align}
\mu[f_1\oplus f_2]&= \mu[f_1]+\mu[f_2] \label{eq:part1}\\
\tau[f_1\oplus f_2] &= \tau[f_1]\oplus\tau[f_2]. \label{eq:part2}
\end{align}
\end{lemma}

\begin{proof}The dependencies $\delta = \dgr{f}$ decomposes  $\delta=\delta_1\delta_2$ in two disjoint graphs $\delta_1 = \dgr(f_1)$, $\delta_2=\dgr(f_2)$. Since $\mu$ is connected,
 $\hom(\mu,\delta) = \hom(\mu, \delta_1)\cup\hom(\mu, \delta_2)$ and the sum in~\eqref{eq:eldiff3}  splits accordingly, yielding~\eqref{eq:part1}. Similarly $\hom(\tau, \delta) = \hom(\tau, \delta_1)\cup\hom(\tau,\delta_2)$ yields~\eqref{eq:part2}.
\end{proof}


As opposed to trees and molecules, aromatic trees (which by \autoref{lem:product_decomposition} are products of molecules and a tree) do not preserve the structure of partitioned vector fields. 
This is the key to the characterisation of B-series. 
In the special case, however, when the partition represents an injection of a vector field in a higher dimensional space, aromatic trees do preserve the partitioned structure.
\begin{lemma}
	\label{prop:eldiffcompat}
	If $f = g\oplus 0$ with $f \in \pvf[n]$ and $g \in \pvf[m]$, then
	\begin{align}
		\eldappl{\gamma}{g \oplus 0}[n] = \eldappl{\gamma}{g}[m] \oplus 0
		\qquad \forall\,\gamma\in\spangle{\Graph[1][k]}
		.
	\end{align}
\end{lemma}
\begin{proof}
The result follows from the elementary differential formula \eqref{eq:eldiff2b}.
We have to prove that the term corresponding to $\nu \colon \Vertices(\gamma) \to [n]$ is zero whenever $\nu\paren[\big]{\Vertices(\gamma)}\not\subset [m]$.
This is clear since there is then a vertex $v$ of $\gamma$ such that $\nu(v) \in [n] \setminus [m]$, and we use that $f^{\nu(v)}_J = 0$ for any derivative $J$.
\end{proof}

\subsection{Dual vector fields}

In classical B-series theory, results on linear independence of elementary differentials are obtained by specially constructed vector fields. 
We use the same technique as in~\cite{butcher1972algebraic,IsQuTs07} to construct such vector fields.

	 First, we need to define the symmetry of a graph.
	 Let $\sigma(\gamma)$ denotes the number of symmetries of a graph, defined as the size of the automorphism group of the graph, i.e, 
	\begin{align}
	\sigma(\gamma) \coloneqq \#{\operatorname{Aut}(\gamma)}
\end{align}
where
\begin{align}
		     \operatorname{Aut}(\gamma) \coloneqq \setc[\big]{\nu\in\operatorname{Aut} \paren[\big]{\Vertices(\gamma)}}{\nu\paren[\big]{\Edges(\gamma)} = \Edges(\gamma)}.	\end{align}

We now define the \demph{labeling} of graph $\gamma \in \Graph$ as a bijection $\lambda \colon \bracket[\big]{\abs{\gamma}} \to \Vertices(\gamma)$.
By convention, we number the roots first.
In particular, \emph{for trees, the root will have number one}.
Incidentally, a similar labeling is chosen in the proof of \cite[Theorem~7.3]{mk-verdier}.

In the rest of this section we choose one fixed labelling for all aromatic forests.
Identifying $\Vertices(\gamma) \equiv [\abs{\gamma}]$ using this labeling, we define, for $\delta\in \Graph[r][n]$, the polynomial vector field $\fv_\delta\in \pvf[n]$ by 
\begin{equation}
	\label{eq:defdualvec}
\fv_\delta^j(x) \coloneqq \frac{1}{\paren[\big]{\sigma(\delta)}^{1/\abs{\delta}}} \prod_{i\in \Parents(j)}x_i ,
\end{equation}
where an empty product is defined as 1. 
By construction, $\delta = \dgr(f_\delta)$.

As an example, consider the following labeled aromatic molecule
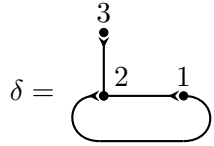
\begin{align}
\delta = \begin{tikzpicture}[baseline=(tree1.south)]
	\begin{scope}[etree, scale=1.5]
	\placeroots{2}
	\children[1]{child{node(up){}}}
	\jointrees{1}{2}
\end{scope}
\node[above] at (up) {3};
\node[above right] at (tree1) {2};
\node[above] at (tree2) {1};
\end{tikzpicture}
\end{align}
Then
\begin{align}
\fv_\delta(x) = \begin{pmatrix}x_2\\x_1x_3\\1\end{pmatrix}.
\end{align}

\begin{lemma} If $\delta = \delta_1\mul\delta_2$ then
\begin{align}
	\fv_\delta = \fv_{\delta_1}\oplus\fv_{\delta_2} 
	.
\end{align}
\end{lemma}

\begin{proof} Clear from the definition of $\fv_\delta$.
\end{proof}

\begin{lemma}\label{le:eldual}
	Let $\mu,\mu'\in M\subset \Graph[0]$ be aromatic molecules and $\tau,\tau'\in T\subset\Graph[1]$ trees. 
	The  elementary differentials of $\mu$ and $\tau$ applied to $f_{\mu'}$ and $f_{\tau'}$  are given by
	\begin{subequations}
	\begin{align}
		\mu[f_{\tau'}] &= 0\label{eq:eld1}\\
		\mu[f_{\mu'}] &= \begin{cases} 1 & \text{if $\mu=\mu'$}\\0 &\text{otherwise} \end{cases}\label{eq:eld2}\\
		\tau[f_{\mu'}] &= 0\label{eq:eld3}\\
		\tau[f_{\tau'}]^1 &= \begin{cases} 1 & \text{if $\tau=\tau'$}\\ 0 &\text{otherwise} \end{cases}\label{eq:eld4}
	\end{align}
\end{subequations}
\end{lemma}

\begin{proof}
Consider two connected graphs $\gamma,\gamma'\in \Graph$.
Define by $f$ the product in \eqref{eq:defdualvec} for the graph $\gamma'$, i.e, $f \coloneqq {\paren[\big]{\sigma(\delta)}^{1/\abs{\delta}}} f_{\gamma'}$.
Let $\Parents$ and $\Parents'$ denote the parent functions in the graphs $\gamma$ and $\gamma'$. From~(\refeq{eq:eldiff2b}) we find 
\begin{align}\gamma[\fv] &= \sum_{\nu\in\hom(\gamma,\gamma')}\prod_{v\in \Vertices(\gamma)} \frac{\partial^{|\Parents(v)|}}{\partial x_{\nu(\Parents(v))}}(\fv)^{\nu(v)}\prod_{r\in \Roots(\gamma)}\partial_{\nu(r)}\\
 &= \sum_{\nu}\prod_{v\in \Vertices(\gamma)} \frac{\partial^{|\Parents(v)|}}{\partial x_{\nu(\Parents(v))}}\prod_{i\in \Parents'(\nu(v))}x_i\prod_{r\in \Roots(\gamma)}\partial_{\nu(r)} \\&= \sum_\nu\prod_{v\in \Vertices(\gamma)}\prod_{i\in \Parents'(\nu(v))\backslash\nu(\Parents(v))}x_i \prod_{r\in \Roots(\gamma)}\partial_{\nu(r)}.
\end{align}
If $\nu\in \hom(\gamma,\gamma')$ sends more than one edge in $\gamma$ to the same edge in $\gamma'$, the expression becomes 0, so it is enough to consider graph embeddings, $\nu\in(\gamma\hookrightarrow\gamma')$, 
the maps that are injective both on vertices and edges. If some edge in $\gamma'$ is not covered by the image of an edge in $\gamma$, the result is a monomial $\prod_i x_i$ running over all edges not covered by the embedding, which evaluates to 0 at $x=0$. If $\nu$ is a graph isomorphism the expression evaluates to 1. 
Hence, we conclude  that for the root component (numbered one by convention):
\begin{align}\paren*{\gamma[\fv]}^1(0) = \begin{cases}\sigma(\gamma) & \text{if $\gamma=\gamma'$,}\\ 0 &\text{otherwise.} \end{cases}
\end{align}
\end{proof}

\subsection{Aromatic series on dual vector fields}

For regular B-series, \autoref{le:eldual} provides a dual basis to the elementary differential.
For aromatic series we must take into account polynomial relations, such as $(\mu\mu)[f_{\mu}] = \paren[\big]{\mu[f_{\mu}]}^2$ for $\mu\in M$. 
The goal of this section is thus to construct the equivalent of the vector fields of \autoref{le:eldual}, but for aromatic trees.
As we shall see, we cannot achieve a corresponding result, but a result that suffices for our purpose.

We need an elementary result first.
If $\gamma\in\Graph$ is disconnected, $\gamma=\gamma_1\mul\gamma_2$, we can decompose $\gamma[f]$ in the following way.

\begin{lemma}\label{le:tpr}For $\gamma_1\in \Graph[r1]$, $\gamma_2\in \Graph[r2]$ we have
	\begin{equation}\label{eq:tensorprod}
		\eldappl{\gamma_1\mul\gamma_2}{f} = \eldappl{\gamma_1}{f}\cdot\eldappl{\gamma_2}{f} 
	\end{equation}
	where the product on the right denotes the symmetric tensor product.
\end{lemma}
\begin{proof}
	If the graph $\gamma = \gamma_1\mul\gamma_2$ is disconnected then $\Parents(v_1)\subset \Vertices(\gamma_1)$ for all $v_1\in \Vertices(\gamma_1)$ and similar for $\gamma_2$, and the
	result is readily checked from~(\refeq{eq:eldiff2b}).
\end{proof}
Note that we will only use that result for products of graphs in $\Graph[0]$, i.e., products of molecules, or products of elements in $\Graph[0]$ and $\Graph[1]$.
In particular, the tensor product on the right will always be either a scalar or a vector.

We now come to the central result of this section.

\begin{lemma}\label{prop:monomial}
	Fix aromatic molecules $\mu_1,\ldots,\mu_m \in \Tc$, scalars $\lambda_1,\ldots,\lambda_m\in\RR$, and a tree $\tau\in\Tree$.
	Define \begin{align}f \coloneqq f_{\tau}\oplus \lambda_1 f_{\mu_1}\oplus\ldots\oplus \lambda_mf_{\mu_m}.\end{align}
	Let $\pi$ be the projection on the first components, that is, $\pi f = f_{\tau}$.
	Choose an arbitrary element $\gamma \in \Graph[1]$.
	If $\gamma = \mu_1^{p_1}\cdots \mu_m^{p_m} \tau$ for some integers $p_1,\ldots,p_m \geq 0$, then
	\begin{align}
	\pi \paren[\big]{\eldappl{\gamma}{f}} = \lambda_1^{\abs{\mu_1}p_1}\cdots \lambda_m^{\abs{\mu_m}p_m} \partial_1
	,
	\end{align}
	otherwise,  $\pi \paren[\big]{\eldappl{\gamma}{f}} = 0 $.
\end{lemma}
\newcommand*\ftaumu{\fv}
\begin{proof} 
	An aromatic tree $\gamma$ can always be written as
	\begin{align}
		\gamma = \sigma \mu_1^{p_1}\cdots \mu_m^{p_m} \tau'
	\end{align}
	for some integers $p_i \geq 0$, an element $\sigma\in\Graph[0]$ which does not contain any of the molecules $\mu_i$, and a regular tree $\tau'\in \Tree$.

First, using \autoref{le:tpr} we obtain that
\begin{align}
	\label{eq:gammamul}
	\gamma[\ftaumu] = \sigma[\ftaumu] \mul (\mu_1[\ftaumu])^{p_1} \cdots ( \mu_m[\ftaumu])^{p_m} \mul \tau'[\ftaumu]
\end{align}
Now, using \autoref{le:parteldiff} and that $\mu$ is $\abs{\mu}$-linear, we obtain that for any molecule $\mu \in \Tc$:
\begin{align}
	\label{eq:muplus}
	\mu[\ftaumu] = \dappl{\mu}{\fdual{\tau}} + \lambda_1^{\abs{\mu}}\dappl{\mu}{\fdual{\mu_1}} + \cdots + \lambda_m^{\abs{\mu}}\dappl{\mu}{\fdual{\mu_m}}
.
\end{align}
Using \eqref{eq:muplus} with $\mu=\mu_i$, we obtain using \autoref{le:eldual} that $\dappl{\mu_i}{\ftaumu} = \lambda_i^{\abs{\mu_i}}$.
If $\sigma$ is not empty, it contains one molecule $\mu$, which, by assumption is distinct from any of the $\mu_i$, and \eqref{eq:muplus} is then zero.
As \eqref{eq:muplus} factors $\dappl{\sigma}{\ftaumu}$ which in turn factorises \eqref{eq:gammamul}, the whole expression \eqref{eq:gammamul} is zero.
If $\sigma$ is empty, we have by convention that $\sigma[\ftaumu] = 1$.
Similarly, we obtain from \autoref{le:parteldiff} that
\begin{align}\dappl{\tau'}{\ftaumu} = \dappl{\tau'}{\fdual{\tau}} \oplus \lambda_1^{\abs{\tau'}}\dappl{\tau'}{\fdual{\mu_1}}\oplus \cdots \oplus \lambda_m^{\abs{\tau'}}\dappl{\tau'}{\fdual{\mu_m}}
	.
\end{align}
We conclude that if $\tau' \neq \tau$ then the root component (which, by convention, has label $1$) of \eqref{eq:gammamul} is zero.
If $\tau' = \tau$ then $\dappl{\tau'}{\ftaumu}^{\rootcomp} = \dappl{\tau}{\ftaumu}^{\rootcomp} = 1$, which concludes the proof.
\end{proof}

\section{Proof of the core result}
\label{sec:mainproof}

We now set out to prove what is the core result of this paper.
Indeed, the following result is the main ingredient in the proof of \autoref{thm:main}.
\begin{theorem}
	\label{prop:bijection}
	For any degree $k$, $\eldiff$ induces a bijection between elements of $\spangle{\Tree_k}$ and equivariant sequences of $k$-forms.
\end{theorem}

\autoref{prop:bijection} is proved through \autoref{prop:injective} (injectivity), \autoref{prop:bseriesequi} (compatibility), and \autoref{prop:surjection} (surjectivity).

\subsection{Injectivity}

\begin{proposition}
	\label{prop:injective}
	The elementary differential map $\eldiff$ is injective.
\end{proposition}
\begin{proof}
	Suppose that $\gamma\in\spangle{\Graph[1][k]}$ and that $\eldiff(\gamma) = 0$.
	Then we have in particular $\eldiff_k(\gamma) = 0$ and \autoref{prop:surjapoly} yields $\gamma = 0$.
\end{proof}

\subsection{B-series are equivariant sequences}

In this subsection we prove the following result.
\begin{proposition}
	\label{prop:bseriesequi}
	$\eldiff$ maps $\spangle{\Tree_k}$ to the space of equivariant sequences of $k$-forms.
\end{proposition}

\newcommand*\Alg{\mathcal{A}}

The geometry of affine spaces is closely related to the existence of a flat, torsion-free connection. 
Note that for each $n\in \NN$ the following \demph{connection} $\tr_n\colon \pvf[n]\times\pvf[n]\rightarrow\pvf[n]$ is well defined because its result is also a polynomial vector field:
\begin{align}
	(f\tr_n g)(x) \coloneqq \left.\frac{\partial}{\partial t}\right|_{t=0}g\big(f(x+tf(x)\big) .
\end{align}

As before, we consider now a connection $\tr$ as a \emph{sequence} of the connections on all dimensions $n$.
We also omit the dimension when the context is clear, so we write:
\begin{align}
f \tr g \coloneqq f \tr_n g \qquad f,g \in \pvf.
\end{align}

\begin{lemma}\label{lem:connectequivar}The connection $\tr$ is an equivariant sequence in the following sense:
	\begin{align}f\affrel[A] \widetilde{f} \text{ and } g\affrel[A] \widetilde{g}\quad\implies\quad \paren*{f\tr g}\affrel[A] \paren*{\widetilde{f}\tr\widetilde{g}}
		\qquad
		A\in\Lin(\RR^n,\RR^m)
		.
	\end{align}
\end{lemma}

\begin{proof} Let $y  = A x$. Then
	\begin{align}
		A(f\tr g)(x) &= \left.\frac{\partial}{\partial t}\right|_{t=0} Ag\paren[\big]{x+tf(x)} = \left.\frac{\partial}{\partial t}\right|_{t=0} \widetilde{g}\paren[\big]{Ax+tAf(x)} \\ &= \left.\frac{\partial}{\partial t}\right|_{t=0}\widetilde{g}\paren[\big]{y+t\widetilde{f}(y)} = (\widetilde{f}\tr\widetilde{g})(y).
	\end{align}
\end{proof}

In the language of algebra, $\set{\pvf[n],\tr_n}$ is an example of a \emph{pre-Lie algebra} \cite{CaEFMa11}, i.e., a vector space with a bilinear binary operation $\tr = \tr_n$ that is neither commutative nor associative, but satisfies the pre-Lie relation
\begin{align}
	f\tr(g\tr h) - (f\tr g)\tr h = g\tr(f\tr h) - (g\tr f)\tr h .
\end{align}

Recall that $\Tree$ denotes the set of all rooted trees, that for $\tau\in T$, $|\tau|$ denote the number of vertices in $\tau$ and that $\Tv$ is the free $\RR$-vector space over $T$.

The \emph{free pre-Lie algebra}, denoted $\set{\Tv,\tr}$, is defined by the binary operation $\tr\colon \Tv\times\Tv\rightarrow\Tv$ given by \emph{grafting} on trees. 
That is, the binary operation $\tr$ given by summing over all trees resulting from attaching successively the tree $\tau_1$ to each vertex of $\tau_2$: 
\begin{equation}
\label{treePL}
	\tau_1  \tr \tau_2 := \sum_{v \in V(\tau_2)} \tau_1 \circ_{v} \tau_2,
\end{equation}
where $ \tau_1 \circ_{v} \tau_2$ denotes attachment of the root of $\tau_1$ to the vertex $v$ of $\tau_2$ via a new edge.
The free pre-Lie algebra is \emph{universal} in the category of pre-Lie algebras:
\begin{proposition}[\cite{ChLi01}]
\label{prop:freeprelie}
For any pre-Lie algebra $\set{\Alg,\tr}$ and any $f\in \Alg$, there exists a unique map $\eldfappl{\cdot}{f}\colon \Tv\rightarrow \Alg$ defined by linearity and the recursion
\begin{subequations}
	\label{eq:frecur}
\begin{align}
	\eldfappl{\oneroot}{f} &= f \label{eq:frecur1}\\
	\eldfappl{\tau_1\tr\tau_2}{f} & = \eldfappl{\tau_1}{\fv}\tr\eldfappl{\tau_2}{\fv},
	\qquad\tau_1,\tau_2\in \Tv
	\label{eq:frecur2}
	.
\end{align}
\end{subequations}
\end{proposition}
When $\Alg = \pvf[n]$ and $\tau\in \Tree$, the elements $\eldfappl{\tau}{f}\in \pvf[n]$ are thus the {elementary differentials} that we defined in \autoref{sec:eldiff}, as the recursion equations \eqref{eq:frecur} are fulfilled in that case.

The following result establishes a proof of \autoref{prop:bseriesequi}.

\begin{proof}[Proof of \autoref{prop:bseriesequi}]
	By \autoref{prop:freeprelie}, we can express the elementary differential map $\eldiff$ using connections and the tree $\oneroot$.
	As $\eldfappl{\oneroot}{\cdot}$ is obviously an equivariant sequence, and as the connection is equivariant in the sense of \autoref{lem:connectequivar}, we conclude that $\eldfappl{\tau}{\cdot}$ also is, for any tree $\tau \in \Tree$.
\end{proof}

\subsection{Surjectivity}
The main result of this subsection is the following result.
\begin{proposition}
	\label{prop:surjection}
	$\eldiff$ is a surjection from $\spangle{\Tree_k}$ to the space of equivariant sequences of $k$-forms.
\end{proposition}

\begin{proof}
The proof contains two steps: assuming $P$ is an equivariant sequence of $k$-forms, we prove that there exists a  $\gamma\in\spangle{\Graph[1][k]}$ such that $P = \eldiff(\gamma)$ (\autoref{prop:Asurjective}), then we prove that $\gamma$ must in fact be in $\spangle{\Tree_k}$ (\autoref{prop:Bsurjective}).
\end{proof}

\begin{lemma}
	\label{prop:Asurjective}
	$\eldiff$ is a surjection from $\spangle{\Graph[1][k]}$ to the space of equivariant sequences of $k$-forms.
\end{lemma}
\begin{proof}
	Let $P$ be an equivariant sequence of $k$-forms.
	For any $n$, using \autoref{prop:surjapoly} we have $\gamma_n\in\spangle{\Graph[1][k]}$ such that $\eldiff_n(\gamma_n) = P_n$.
	We aim to show that $P = \eldiff(\gamma_k)$.
	\begin{enumerate}
\item
	We first show that for $m\leq n$, we have
\begin{align}
	\label{eq:phimn}
	\eldiff_m(\gamma_m) = \eldiff_m(\gamma_n).
\end{align}
   	Indeed, take a $g\in \pvf[m]$ and construct $f = g\oplus 0$.
	The equivariance property \autoref{prop:blockequi} implies that $P_n[g\oplus 0] = P_m[g] \oplus 0$, which gives $\eldiff_n(\gamma_n)[g \oplus 0] = \eldiff_m(\gamma_m)[g] \oplus 0$.
	Using \autoref{prop:eldiffcompat} we obtain $\eldiff_m(\gamma_n)[g] \oplus 0 = \eldiff_m(\gamma_m)[g] \oplus 0$, so \eqref{eq:phimn} is proved.
\item
	In particular, for $n \leq k$, \eqref{eq:phimn} gives $\eldiff_n(\gamma_k) = \eldiff_n(\gamma_n) = P_n$.
	For $k\leq n$ \eqref{eq:phimn} gives $\eldiff_k(\gamma_k) = \eldiff_k(\gamma_n)$, but as $\gamma_n$ is a $k$-form, we can use \autoref{prop:surjapoly} and obtain  $\gamma_k = \gamma_n$.
	For $n \geq k$, this gives $P_n = \eldiff_n(\gamma_k)$.
\end{enumerate}
	We have shown that $\eldiff_n(\gamma_k) = P_n$ for any $n$, so we conclude that $P = \eldiff(\gamma_k)$.
\end{proof}


\begin{lemma}
	\label{prop:Bsurjective}
	Let $P$ be an equivariant sequence of $k$-forms. 
	If $\eldiff(\gamma) = P$ for  $\gamma\in \spangle{\Graph[1][k]}$, then $\gamma \in \spangle{\Tree_k}$.
\end{lemma}

\begin{proof}
	Fix a tree $\tau\in T$.
	The element $\gamma \in \spangle{\Graph[1][k]}$ can be written as
	\begin{align}
		\gamma = p(\mu_1,\ldots,\mu_m) \tau + \gamma'
	\end{align}
	where $\gamma'$ does not contain any occurrence of the tree $\tau$,
	and $p$ is a polynomial over some molecules $\mu_1,\ldots,\mu_n$.
	For instance, if $\gamma = \paren[\big]{3(\mu_1)^2 \mu_2 + (\mu_1)^3} \tau + \gamma'$, then $p(X_1,X_2) = 3X_1^2 X_2 + X_1^3$.
	The goal is now to prove that $p$ is constant.
	Recall the notations of \autoref{prop:monomial}, in particular the special vector $f\coloneqq f_{\tau}\oplus \lambda_1 f_{\mu_1}\oplus\ldots\oplus \lambda_m f_{\mu_m}$ and the projection $\pi$ which projects on the first components of the decomposition of $f$.
	Using \autoref{prop:monomial} we have
	\begin{align}
		\label{eq:plambda}
		\pi \eldappl{\gamma}{f}[N] = p(\lambda_1^{\abs{\mu_1}}, \ldots, \lambda_m^{\abs{\mu_m}}) \partial_1 
		.
	\end{align}
	Now, using that the sequence $\eta$ is equivariant, that is, using \autoref{prop:blockequi} for $f = f_{\tau}\oplus g$, one obtains $\pi \eldappl{\gamma}{f} = \eldappl{\gamma}{f_{\tau}}[\abs{\tau}] $.
	Now, using \autoref{prop:monomial}, we obtain $\eldappl{\gamma}{f_\tau}[\abs{\tau}] = p(0,\ldots,0) \partial_1$, so
	\begin{align}
		\label{eq:pzero}
		\pi \eldappl{\gamma}{f}[N] = p(0,\ldots,0) \partial_1
		.
	\end{align}
	We deduce from \eqref{eq:pzero} and \eqref{eq:plambda} that as $p(\lambda_1^{\abs{\mu_1}},\ldots,\lambda_m^{\abs{\mu_m}}) = p(0,\ldots,0)$, the polynomial $p$ must be constant.
	Finally, we conclude that $\gamma$ is in $\spangle{\Tree_k}$.
\end{proof}

\section{Transfer argument}\label{sec:trans}

\subsection{Transfer to the Taylor terms}

When dealing with B-series, one treats each order separately, so that statements are stable with respect to truncations (or, put differently, with respect to the inverse limit topology).
The terms in a B-series correspond to the terms in the Taylor expansion of a map, so statements about B-series implies statements about Taylor terms and vice versa.
Notice, however, that \autoref{thm:main} says something about the family $\phi$ as a whole, not about its Taylor terms.
Therefore, to prove \autoref{thm:main} through \autoref{prop:bijection}, we need to show that a Taylor term of an  affine equivariant \intmap is also an affine equivariant \intmap; we call this a \emph{transfer argument}.

We first show that affine equivariant sequences preserve fixed points:
\begin{lemma}
	\label{prop:equilocal}
	If $\phi$ is an affine equivariant sequence of smooth maps, then for any integer $n$, any point $x\in\RR^n$, and any  vector field $f \in \rvf{n}$, we have 
	\begin{equation}
		f(x) = 0 \implies {\phi_n}\paren{f}(x) = 0
	.
\end{equation}
	This implies in particular $\phi_n(0) = 0$.
\end{lemma}
\begin{proof}
	Consider the trivial affine injection $\RR^0 \to \RR^n$ defined by $a(0) = x$.
	The mapping $a$ intertwines the zero vector field on $\RR^0$ and $f$, so by equivariance of $\phi_n$, we have: ${\phi_n}\paren{f}(x) =  0$.
\end{proof}

Let $\phi_n\colon \rvf{n}\to\rvf{n}$ be a smooth map, and consider, as in~\autoref{sec:main}, the $k$:th Taylor coefficient $D^k \phi_n(0)$.
Recall that this is a symmetric multi-linear, vector valued map.
The corresponding homogeneous polynomial $\Delta_k(\phi_n)\colon\rvf{n}\to\rvf{n}$ is
\begin{align}
	\Delta_k(\phi_n)(f) \coloneqq D^k\phi_n(0)[f,\ldots,f]
	.
\end{align}

The following transfer argument is the main result of this section.
\begin{proposition}
	\label{pro:transfer}
	Let $\phi=\set{\phi_n}_{n\in\NN}$ be an affine equivariant sequence of smooth  maps. 
	Then, for any fixed $k\in\NN$, the family of maps $\set{\Delta_k(\phi_n)}_{n\in\NN}$ is an affine equivariant sequence of homogeneous polynomials of degree $k$.
\end{proposition}

\begin{proof}
	First, by \autoref{prop:equilocal}, we have that $\phi_n(0) = 0$.
The Taylor polynomial $\Delta_k(\phi_n)$ for a smooth map $\phi_n\colon\rvf{n}\to\rvf{n}$ such that $\phi_n(0) = 0$ is given by \cite[\S5.11]{KrMi97}
\begin{align}
	\label{eq:polyder}
	\Delta_k(\phi_n)(f) = \partial_{t_1}\cdots\partial_{t_k} \phi_n((t_1 + \cdots + t_k)f)|_{t_1=0,\ldots,t_k=0}
	.
\end{align}

Consider two vector fields $f\in\rvf{n}$ and $g\in\rvf{m}$, related by an affine map $a(x) = A x + b$, i.e.,
\(
	f \affrel g
\).
Then $(\lambda f) \affrel (\lambda g)$ for any $\lambda\in\RR$.
Since $\phi$ is an  equivariant sequence, we have $\phi_n(\lambda f) \affrel \phi_m(\lambda g)$, 
which, by its definition \eqref{eq:defrelated}, means $\phi_m(\lambda g) (a x) = A \phi_n(\lambda f)(x)$ for $x\in\RR^n$.
Therefore, taking $\lambda = t_1 + \cdots + t_n$, we obtain
\begin{align}
	\phi_m\paren[\big]{(t_1 + \cdots + t_k)g}(a x) = A \phi_n \paren[\big]{(t_1 + \cdots + t_k)f}(x)
	.
\end{align}
We conclude, using the defining property of $\Delta_k$ in \eqref{eq:polyder}, that $\Delta_k(\phi_n)(f) \affrel \Delta_k(\phi_m)(g)$.
\end{proof}

\subsection{Extension Principle}

We now observe that by \autoref{prop:equilocal}, and because we have polynomials instead of nonlinear maps, we can apply Peetre's theorem and assert that $\phi_n$ depends on a finite number of derivatives of the vector field.
Furthermore, by equivariance at each dimension, we can
restrict the study to that of (linear) equivariant sequences of $k$-forms on polynomial vector fields as defined in \autoref{sec:equiseq}.

We only give a sketch of the proof, as the details are similar to \cite{mk-verdier}, the main novelty being using sequences instead of maps.

\begin{proposition}
	\label{prop:transferpoint}
	There is a bijection between the space of affine equivariant sequences of homogeneous polynomials of degree $k$ on vector fields, and the space of (linear) equivariant sequences of $k$-forms on polynomial vector fields.
\end{proposition}
\begin{proof}
	Observe that \autoref{prop:equilocal} implies in particular that $\phi_n$ is support non-increasing, or \emph{local}, (see \cite[\SS\,2.3]{mk-verdier}) so we can use the extension principle in \cite[Proposition~4.2]{mk-verdier}
	We thus obtain for each dimension $n$ a map $\lfunc_n$ defined from polynomial vector fields to $\RR^n$; moreover, this map is $\GL[n]$-equivariant.
The relation between $\func$ and $\lfunc$ is $\func(f)(x) = \lfunc\paren[\big]{T(f)}(0)$, where $T$ is the Taylor development of $f$.
Note that $\lfunc$ has finite order, i.e., it only needs the Taylor development up to some finite order $k$.

	Consider an affine equivariant sequence $\func_n$ of local $k$-forms on vector fields.
	We thus obtain a sequence $\lfunc_n$ of $\GL[n]$ equivariant maps.
	As the sequence $\func_n$ is affine equivariant, we obtain that the sequence $\lfunc_n$ is equivariant in the sense of \autoref{sec:equiseq}.

	On the other hand, given an equivariant sequence in the sense of \autoref{sec:equiseq}, we obtain a sequence of affine-equivariant maps $\func_n$.
	It is then straightforward to check that this sequence is in fact an equivariant sequence.
\end{proof}

\section*{Acknowledgements}

Klas Modin is supported by the Swedish Research Council (contract VR-2012-335).
Olivier Verdier is supported by the J.C.~Kempe memorial fund (grant number SMK-1238).
Robert McLachlan is supported by the Marsden Fund of the Royal Society of New Zealand.

We express our greatest gratitude to the referees for their numerous useful remarks, in particular for noticing that one can avoid the assumption of locality.
We are also grateful to Alexander Schmeding for giving us the right functional analytic framework of compactly supported vector fields in the definition of an integrator.

\bibliographystyle{plainnat}
\bibliography{ref}

\end{document}